\documentclass[12pt]{article}
\usepackage[margin=1in]{geometry}
\usepackage{amsfonts}
\usepackage{amsmath}
\usepackage{stmaryrd}
\usepackage{amssymb}
\usepackage{graphicx}
\usepackage{appendix}
\usepackage{subfigure}
\usepackage{url}
\usepackage{hyperref}
\usepackage{enumerate}

\newtheorem{theorem}{Theorem}[section]
\newtheorem{cor}{Corollary}[theorem]

\newtheorem{lemma}{Lemma}[section]

\newtheorem{proposition}{Proposition}[lemma]

\makeatletter
\newcommand*{\rom}[1]{\expandafter\@slowromancap\romannumeral #1@}
\makeatother

\title{Topological Equivalence and Curvature Convergence: B\'ezier Surface Approximation}
\author{J.Li\thanks{Department of Mathematics,
       University of Connecticut, 196 Auditorium Road, Unit 3009, Storrs, CT 06269.  Telephone: 860-486-3916.
       Email:  {\tt ji.li@uconn.edu.}}}
       
\date{\today}

\begin{document}
\maketitle

\begin{abstract}
A set of control points can determine a B\'ezier surface and a triangulated surface simultaneously. We prove that the triangulated surface becomes homeomorphic and ambient isotopic to the B\'ezier surface via subdivision. We also show that the total Gaussian curvature of the triangulated surface converges to the total Gaussian curvature of the B\'ezier surface. 
\end{abstract}


\vspace{1ex}

\section{Introduction}

A connected compact surface is classified, up to homeomorphism, by the number of boundary components, the orientability, and Euler characteristic. The Gauss-Bonnet theorem provides a remarkable relation between Euler characteristic (a topological invariant) of a compact surface and the integral of its curvature (an intrinsic invariant). The integral of curvature is said total Gaussian curvature. The concept of total Gaussian curvature has been extended to polyhedral surfaces, and the discrete Gauss-Bonnet theorem has been proved \cite{cheeger1984curvature, reshetnyak1993geometry}. The study of homeomorphic equivalence between smooth and polyhedral surfaces is related to the study of convergence of the total Gaussian curvature during approximation. 

Isotopy is a continuous path of homeomorphisms connecting two given homeomorphisms. In geometric modeling,  isotopy is particularly useful for time-varying geometric models, while homeomorphism is used for static images. 

A two-dimensional B\'ezier surface is a parametric surface defined by an indexed set of control points in space. Because of the simplicity of construction and richness of properties, B\'ezier patch meshes are superior to meshes of triangles as a computational representation of smooth surfaces \cite{G.Farin1990}. In computer graphics, smooth structures are approximated by piecewise linear (p.l.) structures. So B\'ezier surfaces are further rendered using p.l. surfaces by computers. Consequently smooth surfaces are approximated by p.l. surfaces with B\'ezier surfaces as intermediaries. 

The set of control points can be used to determine a triangulated surface, designated as a {\em control surface}. The control surface can then be used as the initial p.l. approximation of the B\'ezier surface. The de Casteljau algorithm is a subdivision process that recursively produces new control surfaces as finer p.l. approximations. 

While previous work regarding surface reconstruction has been done on homeomorphism \cite{amenta2000simple}, there exist some recent papers \cite{Amenta2003, Chazal2005, sakkalis2003ambient} dealing with isotopy. To the best of our knowledge, the equivalence relations (defined by homeomorphism and ambient isotopy respectively) between a B\'ezier surface and the associated triangulated surface presented here, was not previously established. 

\begin{figure}[h!]
\centering
        \subfigure[Non-self-intersecting B\'ezier surface]
   {   \includegraphics[height=5cm]{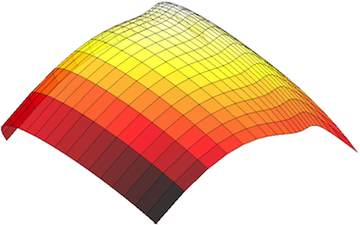}  \label{fig:ob} }
           \subfigure[Self-intersecting control surface]
   {   \includegraphics[height=5cm]{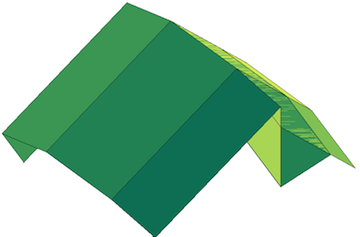}  \label{fig:ppc} }
\caption{Not Homeomorphic}\label{fig:nh}
 \end{figure}

\begin{figure}[h!]
\centering
        \subfigure[Unknotted Bezier torus]
   {   \includegraphics[height=6cm]{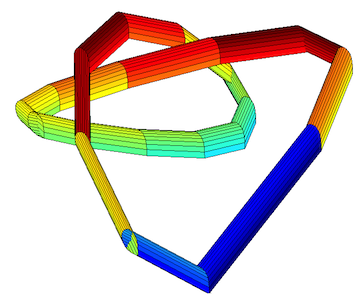}  \label{fig:surf} }
           \subfigure[Knotted Control torus]
   {   \includegraphics[height=6cm]{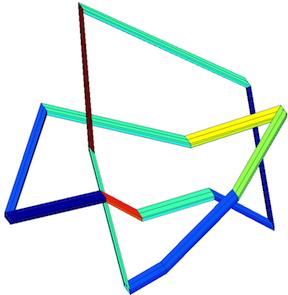}  \label{fig:con} }
\caption{Not Ambient Isotopic}\label{fig:surf-con}
 \end{figure}

Visual examples are given by Figure~\ref{fig:nh} and~\ref{fig:surf-con}. A non-self-intersecting smooth surface in Figure~\ref{fig:ob} is initially defined by a self-intersecting p.l. control surface in Figure~\ref{fig:ppc}. A smooth surface of an unknotted torus in Figure~\ref{fig:surf} is initially defined by a knotted p.l. control surface in Figure~\ref{fig:con}. Our result guarantees that subdivision will produce a non-self-intersecting control surface for the surface in Figure~\ref{fig:ob}, and an unknotted control surface for the surface in Figure~\ref{fig:surf}. The analogue for B\'ezier curves has been established~\cite{bez-iso}.

With the extension of curvature measures to flat spaces, convergence regarding curvature measures from p.l. surfaces to smooth surfaces, and the convergence in the opposite direction, have been studied \cite{brehm1982smooth, cheeger1984curvature, fu1993, meek2000surface}. It finds applications especially for surface reconstruction. In particular, Brehm and Kuhnel \cite{brehm1982smooth} showed that every polyhedral surface can be approximated by smooth surfaces such that the total Gaussian curvature converges. The similar result holds in the opposite direction, i.e.  approximating a smooth surface by polyhedral surfaces \cite{cheeger1984curvature}. As smooth surfaces are often represented by B\'ezier patch meshes, we consider this opposite direction using B\'ezier surfaces and show that the total Gaussian curvature of the triangulated surface will converge to the total Gaussian curvature of the B\'ezier surface.

If two end points of a B\'ezier curve are equal to each other, then the curve is closed. For a B\'ezier surface determined by the control points $\{p_{ij}\}_{i,j=0}^{i=n,j=m}$, if both $p_{i0} = p_{im}$ and $p_{0j} = p_{nj}$ for all $i=0,\ldots,n$ and $j=0,\ldots,m$, then the surface is closed. Otherwise if $p_{i0}  \neq p_{im}$ and $p_{0j} \neq p_{nj}$ for all $i=0,\ldots,n$ and $j=0,\ldots,m$, then the surface is said open. Throughout the paper, we consider B\'ezier surfaces either open or closed\footnote{Without the restriction on endpoints, the topological equivalence may not be obtained.}, with some regularity assumptions which will be specified later. Also, B\'ezier surfaces here are compact and non-self-intersecting\footnote{Throughout the paper, by a non-self-intersecting surface, we mean that the map is injective, except the end points when the surface is closed.}. Now we state our main theorems.

\begin{theorem}
The control surface and a B\'ezier surface will eventually be homeomorphic and ambient isotopic via subdivision.
\end{theorem}

\begin{theorem}
For an open compact surface $\mathbf{b}$, the control surface $\mathbf{l}$ satisfies the following convergence, via subdivision:
$$\sum_{p \in \mathring{\mathbf{l}}} K(p) \rightarrow \int_{\mathring{\mathbf{b}}} K dA+ \int_{\partial\mathbf{b}} \kappa_g -\kappa ds$$
where $\kappa_g$ and $\kappa$ are the geodesic curvature and curvature respectively, at a smooth point of the boundary $\partial \mathbf{b}$. 
\end{theorem}

\begin{theorem}
For a closed B\'ezier surface $\mathbf{b}$, suppose that $\mathbf{l}$ is produced by sufficiently many subdivisions, then we have
$$\sum_{p \in \mathbf{l}} K(p) = \int_{\tilde{\mathbf{b}}} K dA,$$
where $K(p)$ is the total Gaussian curvature at $p \in \mathbf{l}$, $\tilde{\mathbf{b}}$ is a smooth approximation of $\mathbf{b}$, and $K$ is the Gaussian curvature of $\tilde{\mathbf{b}}$.
\end{theorem}

\section{Preliminaries}

A two-dimensional {\em B\'ezier surface} can be defined as a parametric surface where the position of a point $\mathbf{b}$ as a function of the parametric coordinates $u$, $v$ is given by:
$$\mathbf{b}(u,v) = \sum_{i=0}^n \sum_{j=0}^m B_i^n(u) B_j^m(v) p_{ij},$$
evaluated over the unit square, where $B_i^n(u)=\binom{n}{i} u^i (1-u)^{n-i}$ is a Bernstein polynomial, and $\binom{n}{i}=\frac{n!}{i!(n-i)!}$ is the binomial coefficient. A {\em control net} is then formed by connecting the sequence $\{p_{i0}, p_{i1},\ldots,p_{im}\}$ for each fixed $i$ and $\{p_{0j}, p_{1j},\ldots,p_{nj}\}$ for each fixed $j$.

Let $\mathbf{l}(u,v)$ denote the uniform parametrization of the control net over $(\frac{i}{n} \times [0,1]) \cup ([0,1] \times \frac{j}{m})$ for $i=0, 1, \ldots, n$ and $j=0, 1, \ldots, m$. That is, for each $i$ and $j$,
$$\mathbf{l}(\frac{i}{n},\frac{j}{m})=p_{ij},$$
$\mathbf{l}(\frac{i}{n},v)$ is linear for $v \in [\frac{j}{m}, \frac{j+1}{m}]$, and similarly $\mathbf{l}(u,\frac{j}{m})$ is linear for $u \in [\frac{i}{n}, \frac{i+1}{n}]$.

We assume the following regularity: any two control points are not the same except possibly for end points, and for any four adjacent control points $p_{i,j}, p_{i,j+1}, p_{i+1,j}$ and $p_{i+1,j+1}$, any three of them are non-collinear\footnote{This simplifies the following parametrization, and does not impact on isotopy as non-collinearity can be fulfilled by small perturbations which preserve isotopy \cite{L.-E.Andersson2000}.}. Then the region in space determined by these four points consists of two triangles which can be uniformly parametrized. 

A uniform parametrization can be obtained in the following way. Draw the diagonal connecting $p_{i,j}$ and $p_{i+1,j+1}$, and parametrize the line segment $\overline{p_{i,j} p_{i+1,j+1}}$ uniformly, i.e. it interpolates from $\mathbf{l}(\frac{i}{n},\frac{j}{m})$ to $\mathbf{l}(\frac{i+1}{n},\frac{j+1}{m})$ linearly. Connect each point $\mathbf{l}(a,b)$ along $\overline{p_{i,j} p_{i+1,j+1}}$ with the point $\mathbf{l}(a,\frac{j}{m})$ to form a line segment, and uniformly parametrize the line segment. The union of these line segments form the triangle $\triangle{p_{ij}p_{i+1,j}p_{i+1,j+1}}$. Similarly we can obtain the triangle $\triangle{p_{ij}p_{i,j+1}p_{i+1,j+1}}$. All these triangles form a triangulated surface. We designate the union of all such triangles as a {\em control surface}, and denote its parametrization by $\mathbf{l}(u,v)$.

\subsection{Subdivision and properties associated to B\'ezier surfaces}

The de Casteljau algorithm (subdivision) associated to B\'ezier curves and surfaces is fundamental in the curve and surface design, yet it is surprisingly simple \cite{G.Farin1990}. It recursively generates new sets of control points, and divide the curves or surfaces into sub pieces. Each sub piece can be totally defined by a corresponding subset of the control points\footnote{For curves, this subset determines a sub-control polygon}. 

The four sides of a B\'ezier surface, $\mathbf{b}(u,0)$, $\mathbf{b}(u,1)$, $\mathbf{b}(0,v)$ and $\mathbf{b}(1,v)$, are B\'ezier curves, whose control polygons are exactly the sides of the control surface. This fact will be used to study the total curvature of the boundaries of the smooth and triangulated surfaces in Section~\ref{sec:crc}. 

\subsection{Hausdorff distance} 
For a B\'ezier curve, subdivision generates new control polygons more closely approximating the curve under Hausdorff distance \cite{J.Munkres1999}. Analogously for B\'ezier surfaces, subdivision generates new control surfaces more closely approximating the curve under Hausdorff distance. Set 
$$\mathbf{q}_i(v)=\sum_{j=0}^m B_j^m(v) p_{ij},$$
then $$\mathbf{b}(u,v) = \sum_{i=0}^n  B_i^n(u) \mathbf{q}_i(v).$$
For a fixed $v^{\ast} \in [0,1]$, $\mathbf{b}_{v^{\ast}}(u)$ is a B\'ezier curve determined by the control points:
\begin{align*}
\mathbf{q}_0(v^{\ast}) &=\sum_{j=0}^m B_j^m(v^{\ast}) p_{0j}\\
\mathbf{q}_1(v^{\ast}) &=\sum_{j=0}^m B_j^m(v^{\ast}) p_{1j}\\
 & \vdots  \\
\mathbf{q}_n(v^{\ast}) &=\sum_{j=0}^m B_j^m(v^{\ast}) p_{nj}.
\end{align*}
 Let $\mathbf{l}_{v^{\ast}}(u)$ be the p.l. curve of $\mathbf{l}(u, v)$ obtained by fixing $v=v^{\ast}$. We will show $\mathbf{l}_{v^{\ast}}(u)$ converges to $\mathbf{b}_{v^{\ast}}(u)$. It was well known \cite{Nairn-Peters-Lutterkort1999} that the control polygon converges in distance to a B\'ezier curve exponentially (with a rate of $O(\frac{1}{2^k})$ where $k$ is the number of subdivisions). So the polygon $(\mathbf{q}_0(v^{\ast}), \mathbf{q}_1(v^{\ast}), \ldots, \mathbf{q}_n(v^{\ast}))$ converges in distance to the B\'ezier curve $\mathbf{b}_{v^{\ast}}(u)$. Thus, it suffices to show that $\mathbf{l}_{v^{\ast}}(u)$ converges to the polygon $(\mathbf{q}_0(v^{\ast}), \mathbf{q}_1(v^{\ast}), \ldots, \mathbf{q}_n(v^{\ast}))$. Note that $\mathbf{l}_{v^{\ast}}(u)$ is a polygon with vertices $\{\mathbf{l}_{v^{\ast}}(0), \mathbf{l}_{v^{\ast}}(\frac{1}{n}), \ldots, \mathbf{l}_{v^{\ast}}(1)\}$, so it suffices to show that $\mathbf{l}_{v^{\ast}}(\frac{i}{n})$ converges to $\mathbf{q}_i(v^{\ast})$ for each $i=0, 1, \ldots, n$. 

Start from $i=0$. Note that $\mathbf{l}_v(0)$ is a polygon with vertices $\{p_{00}, p_{01}, \ldots, p_{0m}\}$, while $\mathbf{q}_0(v)$ is a B\'ezier curve determined by the same set of points. So $\mathbf{l}_v(0)$ is the control polygon of $\mathbf{q}_0(v)$. Because of the exponential convergence of the control polygon to a B\'ezier curve, $\mathbf{l}_v(0)$  exponentially converges to $\mathbf{q}_0(v)$ for any $v$, and of course, particularly for $v=v^{\ast}$. Similarly $\mathbf{l}_{v^{\ast}}(\frac{i}{n})$  exponentially converges to $\mathbf{q}_i(v^{\ast})$ for each $i=1, \ldots, n$. 

Since $v^{\ast} \in [0,1]$ is an arbitrary fixed $v$-value, we have that $\mathbf{l}(u,v)$ exponentially converges to $\mathbf{b}(u,v)$ for each $(u,v) \in [0,1] \times [0,1]$. This implies the control surface $\mathbf{l}(u,v)$ exponentially converges\footnote{Subdivision is applied in both $u$ and $v$ directions. If it is in only one direction, the convergence fails.} in Hausdorff distance to the B\'ezier surface $\mathbf{b}(u,v)$. 

\subsection{Tangent and normal vectors}
The lemma below follows from that the first discrete derivatives of $\mathbf{l}(u,v)$ converge to the corresponding derivatives of $\mathbf{b}(u,v)$ via subdivision \cite{Morin_Goldman2001}.


\begin{lemma}\label{lem:normal}
The tangent and normal vectors of control surface $\mathbf{l}$ at vertices converge to the corresponding tangent and normal vectors of B\'ezier surface $S$.
\end{lemma}

The {\em tangent bounding cone} of a curve is the smallest direction cone that contains all unit tangent vectors of the curve, denoted by $<a,\theta>$, where $a$ is the axis of the cone, and $\theta$ is the span of the cone.  

For a surface, we consider the normal bounding cone $<a_n, \theta_n>$ that bounds all normal vectors, and the isoparametric tangent bounding cones $<a_u, \theta_u>$ and $<a_v, \theta_v>$ that bound all tangent vectors in the corresponding isoparametric direction. A surface is non-self-intersecting if the following conditions are satisfied \cite{ho2001surface}:
\begin{enumerate}[i]
\item $\theta_n < \pi$,
\item $|a_n \cdot a_u| < \cos \frac{\theta_u}{2}$, and 
\item $|a_n \cdot a_v| < \cos \frac{\theta_v}{2}$.
\end{enumerate}

Specifically for a control surface, we associate each triangle a normal vector that perpendicular to the plane determined by the triangle. The normal bonding cone of the control surface is then defined. 

\section{Topology}

In this section, we show that, via subdivision, the control surface $\mathbf{l}$ becomes
\begin{enumerate}
\item non-self-intersecting and homeomorphic to $\mathbf{b}$, and
\item ambient isotopic to $\mathbf{b}$,
\end{enumerate}
where $\mathbf{b}$ is the underlying compact and non-self-intersecting B\'ezier surface.

\subsection{Homeomorphism}

To establish homeomorphism, we first prove that after sufficiently many subdivisions, $\mathbf{l}(u,v)$ will be injective, except possibly for $(u,v)$ along the boundary of the unit square when $\mathbf{l}(u,v)$ is closed. 

\begin{lemma}\label{lem:de-contl}
After sufficiently many subdivisions, there exists $\delta>0$ such that $\mathbf{l}(u,v)$ is injective for $(u,v) \in D_{\delta}$, where $D_{\delta}$ is a closed disk of radius $\delta$ in $[0,1] \times [0,1]$. 
\end{lemma}
\begin{proof}
We adopt discrete derivatives \cite{Morin_Goldman2001} for the control surface $\mathbf{l}(u,v)$. Consider the discrete derivatives at a point $(\frac{i}{n}, \frac{j}{m})$ in the direction of $u$ and $v$. Denote the derivatives as $\mathbf{l}_u(\frac{i}{n}, \frac{j}{m})$ and $\mathbf{l}_v(\frac{i}{n}, \frac{j}{m})$. Morin and Goldman \cite{Morin_Goldman2001} showed that $\mathbf{l}_u(\frac{i}{n}, \frac{j}{m})$ and $\mathbf{l}_v(\frac{i}{n}, \frac{j}{m})$ converge to the derivatives of the B\'ezier surface, $\mathbf{b}_u(\frac{i}{n}, \frac{j}{m})$ and $\mathbf{b}_v(\frac{i}{n}, \frac{j}{m})$, under subdivision. 

Since $\mathbf{b}$ is smooth, if $\delta$ is sufficiently small, then $||\mathbf{b}_u(u_1,v_1)-\mathbf{b}_u(u_2,v_2)||$ and $||\mathbf{b}_v(u_1,v_1)-\mathbf{b}_v(u_2,v_2)||$ are sufficiently small over $D_{\delta}$. Therefore, provided sufficiently many subdivisions,  $|\mathbf{l}_u(u_1,v_1)-\mathbf{l}_u(u_2,v_2)|$ and $|\mathbf{l}_v(u_1,v_1)-\mathbf{l}_v(u_2,v_2)|$ is sufficiently small over $D_{\delta}$. It follows that the normal and isoparametric tangent bounding cones have sufficiently small spanning angles for the sub-surface of $\mathbf{l}$ corresponding to $(u,v)\in D_{\delta}$. 

Note that the sub-surface of $\mathbf{l}$ corresponding to  the sub-domain $D_{\delta}$ converges to a plane. It follows that $|a_n \cdot a_u|$ and $|a_n \cdot a_v|$ become small enough, where $a_n, a_u$ and $a_v$ are similar as those in the above conditions i, ii and iii. Therefore, the conditions i, ii and iii can be fulfilled, and the conclusion follows. \hfill $\boxempty$\\
\end{proof}

When will consider the non-self-intersection problem mainly for closed surfaces. The proof for open surfaces will follow easily. Suppose that the surfaces are closed, i.e. $\mathbf{l}(u,0)=\mathbf{l}(u,1)$ and $\mathbf{l}(0,v)=\mathbf{l}(1,v)$ for all $u, v \in [0,1]$. Consider a point $\mathbf{l}(u^{\ast},0)$ along $\mathbf{l}(u,0)$, and a point $\mathbf{l}(0,v^{\ast})$ along $\mathbf{l}(0,v)$. Denote the neighborhood of $\mathbf{l}(u^{\ast},0)$ as $U_a$, that is, $U_a$ is the image of 
$$([u^{\ast}-a,u^{\ast}+a] \times [0,a]) \cup ([u^{\ast}-a,u^{\ast}+a] \times [1-a,1]).$$
Similarly denote a neighborhood of $\mathbf{l}(0,v^{\ast})$ as
$$([0,b] \times [v^{\ast}-b,v^{\ast}+b]) \cup ([1-b,1] \times [v^{\ast}-b,v^{\ast}+b]).$$

\begin{lemma}\label{lem:west}
After sufficiently many subdivisions, there exist $a, b>0$ small enough such that $U_a$ and $U_b$ are not self-intersecting, i.e. $\mathbf{l}(u,v)$ for $(u,v)$ being restricted within the subdomains, is injective except at the endpoints. 
\end{lemma}
\begin{proof}
Denote the half of $U_a$ corresponding to $[u^{\ast}-a,u^{\ast}+a] \times [0,a]$ as $U^+_a$, and the other half corresponding to $[u^{\ast}-a,u^{\ast}+a] \times [1-a,1]$ as $U^-_a$. By Lemma~\ref{lem:de-contl}, we can choose $a>0$ small enough such that both $U^+_a$ and $U^-_a$ are not self-intersecting. The sub-surface $U_a$ is obtained by pasting $U^+_a$ and $U^-_a$ together along the common edge. They can not intersect, if $U^+_a$ and $U^-_a$ are flat enough, i.e. the the change of normals is small enough. That is, if the normal bounding cone of $U_a$ has a small spanning angle. The condition can be satisfied by choosing small $a,b$ and sufficiently many subdivisions. Provided this, $U_a$ is not self-intersecting. Similarly for $U_b$. \hfill $\boxempty$
\end{proof}

\begin{proposition}\label{prop:aste}
After sufficiently many subdivisions, there exists $d>0$ such that if $0<|\mathbf{b}(u_1,v_1)-\mathbf{b}(u_2,v_2)|<d$ then $\mathbf{l}(u_1,v_1) \neq \mathbf{l}(u_2,v_2)$, except at the end points.
\end{proposition}
\begin{proof}
Suppose that the surfaces are closed. Recall that  $U_a, U_b$ in Lemma~\ref{lem:west} are the neighborhoods of two end points along $\mathbf{l}(u,0)$ and $\mathbf{l}(0,v)$ respectively. Denote the union of all the neighborhoods of points along $\mathbf{l}(u,0)$ as $S_a$ ($S$ represents a strip), and all the neighborhoods of the points along $\mathbf{l}(0,v)$ as $S_b$. 

Denote the truncated surface $\mathbf{l} \setminus (S_a \cap S_b)$ as $\hat{\mathbf{l}}$, and the corresponding underlying B\'ezier surface as $\hat{\mathbf{b}}$. Suppose that $d>0$ is small enough such that either the following Case 1 or Case 2 holds. 

Case 1: Both $\mathbf{b}(u_1,v_1)$ and $\mathbf{b}(u_2,v_2)$ lie in $\hat{\mathbf{b}}$. (If necessary, choose smaller $U_a, U_b$.) Let $m$ be the minimal separation distance of $\hat{\mathbf{b}}$. Since $\hat{\mathbf{b}}$ is the image of a homeomorphism, there exists an $0<d<m$ such that if $|\mathbf{b}(u_1,v_1) - \mathbf{b}(u_2,v_2)|<d$, then $|(u_1,v_1) - (u_2,v_2)| < \delta$ (where $\delta$ is given in Lemma~\ref{lem:de-contl}) so that the corresponding $\mathbf{l}(u_1,v_1) \neq \mathbf{l}(u_2,v_2)$, by Lemma~\ref{lem:de-contl}.

Case 2: Both  $\mathbf{b}(u_1,v_1)$ and $\mathbf{b}(u_2,v_2)$ lie in $\mathbf{b} \setminus \hat{\mathbf{b}}$, and  either $\mathbf{l}(u_1,v_1), \mathbf{l}(u_2,v_2) \in U_a$ or $\mathbf{l}(u_1,v_1), \mathbf{l}(u_2,v_2) \in U_b$. The the conclusion follows from Lemma~\ref{lem:west}. 

If the surfaces are open, the proof is the same as Case 1. \hfill $\boxempty$
\end{proof}

\begin{lemma}\label{lem:nsint}
After sufficiently many subdivisions, $\mathbf{l}(u,v)$ becomes injective, except possibly at the endpoints.
\end{lemma}

\begin{proof}
By sufficiently many subdivisions, we can have $|\mathbf{l}(u,v)-\mathbf{b}(u,v)|<\frac{d}{2}$ for all $u,v \in [0,1]$, where $d$ is the given by Proposition~\ref{prop:aste}. Assume to the contrary there are $(u_1,v_1), (u_2,v_2) \in [0,1] \times [0,1]$ where $(u_1,v_1) \neq (u_2,v_2)$ but $\mathbf{l}(u_1,v_1)=\mathbf{l}(u_2,v_2)$, then 
$$|\mathbf{b}(u_1,v_1)-\mathbf{b}(u_2,v_2)|=|\mathbf{b}(u_1,v_1)-\mathbf{l}(u_1,v_1)-\mathbf{b}(u_2,v_2)+\mathbf{l}(u_2,v_2)|$$ 
$$\leq |\mathbf{b}(u_1,v_1)-\mathbf{l}(u_1,v_1)| + |\mathbf{b}(u_2,v_2)-\mathbf{l}(u_2,v_2)|<d,$$
which contradicts to Proposition~\ref{prop:aste}. \hfill $\boxempty$
\end{proof}

\begin{lemma}\label{lem:homeom}
For open or closed $\mathbf{b}(u,v)$ and $\mathbf{l}(u,v)$, if they are injective for $(u,v) \in (0,1) \times (0,1)$, then they are homeomorphic.
\end{lemma}

\begin{proof}
Note that the control surface and a B\'ezier surface are simultaneously open or closed. If the surfaces are open, then the homeomorphism trivially holds, as both of the B\'ezier surface and the control surface are homeomorphic to the unit square. 

If the the surfaces are closed, then we have
$$ \mathbf{l}(u,0)= \mathbf{l}(u,1)\ \text{and}\ \mathbf{l}(0,v)=\mathbf{l}(1,v);$$
$$ \mathbf{b}(u,0)= \mathbf{b}(u,1)\ \text{and}\ \mathbf{b}(0,v)=\mathbf{b}(1,v).$$
for all $u, v \in [0,1]$. So both of the control surface and the B\'ezier surface are homeomorphic to the quotient space obtained from the unit square by pasting its opposite edges with the same direction, which is the fundamental polygon of torus.  \hfill $\boxempty$
\end{proof}

\begin{theorem}\label{thm:homeo}
The control surface will eventually be homeomorphic to an open or closed B\'ezier surface $\mathbf{b}$ via subdivision.
\end{theorem}

\begin{proof}
It follows from Lemma~\ref{lem:nsint} and Lemma~\ref{lem:homeom}. \hfill $\boxempty$
\end{proof}
\subsection{Ambient Isotopy}

We prove the ambient isotopy for open surfaces first and then closed surfaces. 

\begin{lemma}\label{smthap} \textup{\cite{brehm1982smooth}}
A compact polyhedral surface\footnote{The smooth approximations in the paper \cite{brehm1982smooth} satisfy not only the properties given here, but also some other curvature properties. The paper \cite{brehm1982smooth} considers polyhedral surfaces without boundary. However the properties given here remain hold for polyhedral surfaces with boundares.} $M$ can be approximated by a sequence of smooth surfaces $\{M_n\}_{n=1}^{\infty}$ such that 
\begin{enumerate}
\item Each $M_n$ is homeomorphic to $M$;
\item $M_n=M$ outside of the $\frac{1}{n}$ neighborhood of the $1$-skeleton of $M$;
\item $M_n \rightarrow M$ as $n \rightarrow \infty$ with respect to the Hausdorff Metric; 
\end{enumerate} 
\end{lemma}

\begin{cor}\label{coro:cpsm}
A compact polyhedral surface $M$ can be approximated by a sequence of smooth surfaces $\{M_n\}_{n=1}^{\infty}$ such that each $M_n$ is ambient isotopic to $M$. 
\end{cor}
\begin{proof}
The smooth surface $M_n$ is obtained by a smoothing in the paper \cite{brehm1982smooth}. Note that as long as the smoothing is within a small scope such that the process, a continuous deformation, does not yield intersections, ambient isotopy is preserved. The smoothing is performed within the $\frac{1}{n}$ neighborhood of the $1$-skeleton of $M$ yielding no intersection \cite{brehm1982smooth}, so the homeomorphism in Lemma~\ref{smthap} can be extended to an ambient isotopy.  \hfill $\boxempty$
\end{proof}

\begin{theorem}
The open control surface $\mathbf{l}$ will eventually be ambient isotopic to an open B\'ezier surface $\mathbf{b}$ via subdivision.
\end{theorem}
\begin{proof}
By Corollary~\ref{coro:cpsm}, let $\mathbf{f}(u,v)$ be ambient isotopic smooth approximation of $\mathbf{l}(u,v)$. Note that for $\forall \epsilon>0$, there exists an embedding $\mathbf{g}(u,v)$ for $u,v \in [0,1]$ such that $|\mathbf{f}(u,v)-\mathbf{g}(u,v)|<\epsilon$ (by an approximation theorem \cite[p.26]{Hirsch}),  and $\mathbf{f}$ and $\mathbf{g}$ are homotopic (by \cite[Lemma 1.5]{Hirsch}). Since all surfaces between $\mathbf{f}$ and $\mathbf{g}$ determined by the homotopy are homeomorphic to the unit square, they are homeomorphic. So the homotopy is actually an ambient isotopy.  There is only one way (up to ambient isotopy) to embed a disk \cite[Theorem 3.1]{Hirsch}, so $\mathbf{b}$ and $\mathbf{g}$ are ambient isotopic. By the equivalence relation of ambient isotopy, we have $\mathbf{b}$ and $\mathbf{l}$ are ambient isotopic. \hfill $\boxempty$
\end{proof}

\begin{theorem}
The closed control surface $\mathbf{l}$ will eventually be ambient isotopic to a closed B\'ezier surface $\mathbf{b}$ via subdivision.
\end{theorem}
\begin{proof}
When $\mathbf{b}(u,v)$ is closed, we use the following theorem: Suppose $S$ and $S'$ are compact orientable surfaces embedded in $\mathbb{R}^3$, and $\mathbf{T}$ is a tubular neighborhood\footnote{A topological thickening used in \cite{Chazal2005} is equivalent to the closure of a tubular neighborhood defined by \cite{DoCarmo1976}.} of $S$. Chazal and Cohen-Steiner \cite{Chazal2005} proved that if $S'$ is homeomorphic to $S$, $S' \subset \mathbf{T}$, and $\mathbf{\bar{T}} \setminus S'$ is disconnected, where $\mathbf{\bar{T}}$ is the closure of $\mathbf{T}$, then $S'$ is ambient isotopic\footnote{In the case of \cite{Chazal2005}, isotopy and ambient isotopy are equivalent.} to $S$. These conditions of the theorem can be fulfilled for $\mathbf{b}$ and $\mathbf{l}$ by sufficiently many subdivisions. \hfill $\boxempty$
\end{proof}
\section{Convergence regarding Curvature}\label{sec:crc}
The total Gaussian curvature $K(p)$ of a vertex $p$ on a polyhedral surface $\Omega \in \mathbb{R}^3$ is defined as 
$$K(p)=2\pi - \sum_{i=1} \theta_i(p),$$
where $\theta_i(p)$ is the interior angle of face $f_i$ at $p$. Let $\chi (\Omega)$ be the Euler characteristic of $\Omega$. It was shown \cite{reshetnyak1993geometry} the following discrete Gauss-Bonet theorem: 
\begin{enumerate}
\item If $\Omega$ is closed, then  \begin{equation}\label{discGBcl} \sum_{p \in \Omega} K(p) =2\pi \chi (\Omega).\end{equation}
\item If $\Omega$ is open, then \begin{equation}\label{discGB}\sum_{p \in \mathring{\Omega}} K(p) + \sum_{p\in \partial \Omega} \alpha(p)=2\pi \chi (\Omega),\end{equation}
where $\mathring{\Omega}$ is the interior of $\Omega$, and $\alpha(p)$ is the exterior angle at a vertex $p$ of the boundary $\partial \Omega$ of $\Omega$.
\end{enumerate}

For open B\'ezier surfaces, we first consider the convergence regarding total curvature of the boundaries. The total curvature of a p.l. curve is defined as the sum of exterior angles \cite{Milnor1950}. The Gauss-Bonet theorem for smooth surface $\Omega$ with a curvilinear boundary is \cite{DoCarmo1976},
\begin{equation}\label{GBsm}\int_{\mathring{\Omega}} K dA + \int_{\partial \Omega} \kappa_g ds + T_{\kappa}=2\pi \chi (\Omega),\end{equation}
where $\kappa_g$ is the geodesic curvature at a smooth point and $T_{\kappa}$ is the total curvature at vertices. 

\begin{lemma}\label{lem:lencon}
Let $\ell(t)$ be the control polygon of a B\'ezier curve $\gamma(t)$ where $t\in[0,1]$, which is uniformly parametrized. Then $\int_0^1 |\ell^m(t)| dt$ converges to the length of the curve $\ell^{m-1}(t)$, via subdivision, where $m \geq 1$ is the order of discrete derivatives. 
\end{lemma}
\begin{proof}
Since $\ell^m(t) \rightarrow \gamma^m(t)$ \cite{Morin_Goldman2001}, we have $\int_0^1 |\ell^m(t)| dt$ converges to $\int_0^1 |\gamma^m(t)| dt$, the length of $\gamma^{m-1}(t)$. Also, the length of $\ell^{m-1}(t)$ converges to the length of $\gamma^{m-1}(t)$ \cite{paeth1995graphics}. It follows that  $\int_0^1 |\ell^m(t)| dt$ converges to the length of the curve $\ell^{m-1}(t)$. \hfill $\boxempty$
\end{proof}

\begin{lemma}\label{lem:taco}
The total curvature of the control polygon converges to the total curvature of a B\'ezier curve. 
\end{lemma}
\begin{proof}
Denote a B\'ezier curve as $\gamma(t)$ and the control polygon as $\ell(t)$, where $t\in[0,1]$. Suppose without loss of generality that $\gamma(t)$ is parametrized by arc length. Let $\alpha_i$ be an exterior angle of $\ell(t)$ and $\sum \alpha_i$ be the total curvature of $\ell(t)$. Then we need to show that, via subdivision,
$$\sum \alpha_i \rightarrow \int_0^1 |\gamma''(t)| dt. $$
Since $\ell''(t) \rightarrow \gamma''(t)$ \cite{Morin_Goldman2001}, it suffices to show that 
$$\sum \alpha_i \rightarrow  \int_0^1 |\ell''(t)| dt. $$
By Lemma~\ref{lem:lencon}, it suffices to show that
$$\sum \alpha_i \rightarrow \text{the length of}\ \ell'(t). $$
Let $u(t)$ be a p.l. curve determined by vertices $\{ \frac{\ell'(t)}{|\ell'(t)|}, \text{for}\ $t$\ \text{where}\ \ell'(t)\ \text{is a vertex}\}$. Since $|\ell'(t)| \rightarrow |\gamma'(t)|=1$, the length of $\ell'(t)$ converges to the length of $u(t)$. But $u(t)$ is inscribed \cite{Milnor1950} in a curve on the unit sphere whose length is $\sum \alpha_i$. So the result follows. \hfill $\boxempty$
\end{proof}

\begin{theorem}
For an open compact surface $\mathbf{b}$, the control surface $\mathbf{l}$ satisfies the following convergence, via subdivision:
$$\sum_{p \in \mathring{\mathbf{l}}} K(p) \rightarrow \int_{\mathring{\mathbf{b}}} K dA+ \int_{\partial\mathbf{b}} \kappa_g -\kappa ds$$
where $\kappa_g$ and $\kappa$ are the geodesic curvature and curvature at a smooth point of the boundary $\partial \mathbf{b}$. 
\end{theorem}

\begin{proof}
Consider the four conner points of $\mathbf{b}$ and $\mathbf{l}$. By the convergence of the first derivatives, the exterior angles at these conner points satisfies the convergence from $\mathbf{l}$ to $\mathbf{b}$. So the total curvature of conner points converges. 

Let $T_{\kappa}^i$ for $i=1,2,3,4$ denote the total curvature of four boundary control polygon of $\mathbf{l}$. It follows from Theorem~\ref{thm:homeo} and Equation~\ref{discGB} and~\ref{GBsm} that 
$$\sum_{p \in \mathring{\mathbf{l}}} K(p) + \sum_{i=1}^4 T_{\kappa}^i= \int_{\mathring{\mathbf{b}}} K dA+ \int_{\partial \mathbf{b}} k_g ds.$$

However, by Lemma~\ref{lem:taco}, 
$$\sum_{i=1}^4 T_{\kappa}^i \rightarrow \int_{\partial \mathbf{b}} k ds.$$
The conclusion follows.  \hfill $\boxempty$\\
\end{proof}

Note that for a closed B\'ezier surface, $\mathbf{b}(u,v)$ is not smooth at the points where common edges are connected, for which the Gaussian curvatures are not well-defined. The surface $\mathbf{b}(u,v)$ can be smoothed at the junction points according to Lemma~\ref{smthap}, satisfying the properties in Lemma~\ref{smthap}. We compare the total Gaussian curvature between the control surface $\mathbf{l}$ and the smooth approximation, denoted as $\tilde{\mathbf{b}}$. 

\begin{theorem}
For a closed B\'ezier surface $\mathbf{b}$, suppose that $\mathbf{l}$ is produced by sufficiently many subdivisions, then we have
$$\sum_{p \in \mathbf{l}} K(p) = \int_{\tilde{\mathbf{b}}} K dA,$$
where $K(p)$ is the total Gaussian curvature at $p \in \mathbf{l}$, $\tilde{\mathbf{b}}$ is a smooth approximation of $\mathbf{b}$, and $K$ is the Gaussian curvature of $\tilde{\mathbf{b}}$.
\end{theorem}
\begin{proof}
It follows from the discrete Gauss-Bonnet theorem given by Equation~\ref{discGBcl}, and the homeomorphism established by Theorem~\ref{thm:homeo}. \hfill $\boxempty$
\end{proof}

\section{Conclusion and Future Work}
We proved that the triangulated surface associated to an open or closed B\'ezier surface will be eventually ambient isotopic to the B\'ezier surface via subdivision. By the Gauss-Bonnet theorem, we showed that the triangulated surface converges to the smooth surface regarding total Gaussian curvature. This may contribute to the theoretical foundation of using B\'ezier surfaces in computer aided geometric design for geometric modeling. For practical potential, it may be worth investigating in the future how many subdivision iterations are needed to obtain the ambient isotopy. Besides, convergence regarding other curvature measures, such as total absolute Gaussian curvature, total mean curvature, and total absolute mean curvature may be of interesting as a future endeavor. 

\section*{Acknowledgments}
The author thanks Professor Thomas J. Peters for stimulating questions that motivate this work, and Anne Berres for her Matlab codes of B\'ezier surfaces.

\bibliographystyle{plain}
\bibliography{ji-tjp-biblio}

\end{document}